\documentclass[pdfmx, a4paper, 11pt]{article}

\title{Realization of symmetry of $A^{(1)*}_2$-surfaces as transformations of logarithmic connections}
\author{Takafumi Matsumoto\footnote{National Institute of Technology, Kitakyushu College, 5-20-1 Shii, Kokuraminamiku, Kitakyushu, Fukuoka, 802-0985 Japan, Email address: tfmatumt@apps.kct.ac.jp}}
\date{}
\date{}

\usepackage[top=30truemm, bottom=30truemm, left=30truemm, right=30truemm]{geometry}
\usepackage{amsmath,amsthm, mathrsfs}
\usepackage{amssymb}
\usepackage{latexsym}
\usepackage{tikz}
\usepackage{tikz-cd}
\usepackage{here}
\usepackage{bm}
\usepackage[all]{xy}

\theoremstyle{definition}
\newtheorem{definition}{Definition}[section]
\newtheorem{theorem}[definition]{Theorem}
\newtheorem{proposition}[definition]{Proposition}
\newtheorem{Problem}{Problem}[section]
\newtheorem{lemma}[definition]{Lemma}

\newtheorem{remark}[definition]{Remark}

\newcommand{\mbZ}{{\mathbb Z}}                   
\newcommand{\mbC}{{\mathbb C}}        
\newcommand{\mbH}{\mathbb{H}}           
\newcommand{\mbP}{\mathbb{P}}
\newcommand{\mbR}{\mathbb{R}}

\newcommand{\mcD}{\mathcal{D}}
\newcommand{\mcE}{\mathcal{E}}
\newcommand{\mcF}{\mathcal{F}}

\newcommand{\mcM}{\mathcal{M}}
\newcommand{\mcN}{\mathcal{N}}
\newcommand{\mcO}{\mathcal{O}}

\newcommand{\mcS}{\mathcal{S}}

\newcommand{\mcU}{\mathcal{U}}
\newcommand{\mcY}{\mathcal{Y}}

\newcommand{\Aut}{\textrm{Aut}}
\newcommand{\Coker}{\textrm{Coker}\,}

\newcommand{\res}{\textrm{res}}

\newcommand{\App}{\textrm{App}}
\newcommand{\Image}{\textrm{Im}\,}
\newcommand{\Ker}{\textrm{Ker}\,}

\newcommand{\id}{\textrm{id}}

\newcommand{\Pic}{\textrm{Pic}\,}

\newcommand{\pl}{\mathbb{P}^1}
\newcommand{\Opl}{\mathcal{O}_{\pl}}

\newcommand{\Ompld}{\Omega^1_{\pl}(D)}
\newcommand{\tmid}{\textrm{mid}}
\newcommand{\mc}{\textrm{mc}}

\newcommand{\GM}{\textrm{GM}}
\newcommand{\MDR}{\textrm{MDR}}
\newcommand{\DR}{\textrm{DR}}
\newcommand{\rMDR}{\textrm{MDR}_{X/Z}}
\newcommand{\rDR}{\textrm{DR}_{X/Z}}

\newcommand{\vb}{\mathrel{} \middle | \mathrel{}}

\newcommand{\xj}[1]{\Omega^{#1}_{X}(\log J)}
\newcommand{\zd}{\Omega^1_{Z}(D_Z)}

\newcommand{\xjzq}[1]{\Omega^{#1}_{(X,J)/(Z, D_Z)}}

\allowdisplaybreaks[1]

\begin{document}
\maketitle
\begin{abstract}
An $A^{(1)*}_2$-surface is a space of initial conditions of certain difference Painlev\'e equations. $A^{(1)*}_2$-surfaces are realized as the moduli spaces of parabolic logarithmic connections. In this paper, we realize the symmetry of $A^{(1)*}_2$-surfaces as transformations of parabolic logarithmic connections. 
\end{abstract}

\section{Introduction}
The (additive) difference Painlev\'e equations are eleven types of non-linear difference equations.
Their classification was carried out by Sakai \cite{SaRa} through the classification of suitable smooth projective rational surfaces by using affine root systems. We call a surface associated with an affine root system $R$ an $R$-surface. In his theory, the extended affine Weyl group $\widetilde{W}(R^\perp)$ acts on a suitable family of $R$-surfaces, and the actions of elements in the root lattice give rise to discrete dynamical systems, which are the discrete Painlev\'e equations. In particular, an $R$-surface is identified as the space of initial conditions of the discrete Painlev\'e equations, and the extended affine Weyl group $\widetilde{W}(R^\perp)$ is the symmetry of the space of initial conditions. Therefore, the symmetries (B\"acklund transformations) of the difference Painlev\'e equations are explained by the symmetry of $R$-surfaces. The following is the list of surfaces and their symmetries corresponding to the difference Painlev\'e equations.
\[
	\centering
	\hspace{-25pt}
	\begin{tabular}{|c||c|c|c|c|c|c|c|c|c|c|c|c|}\hline
		surface &$R$&$A^{(1)**}_0$&$A^{(1)*}_1$&$A^{(1)*}_2$& $D^{(1)}_4$&$D^{(1)}_5$&$D^{(1)}_6$&$D^{(1)}_7$&$D^{(1)}_8$&$E^{(1)}_6$&$E^{(1)}_7$&$E^{(1)}_8$ \\ \hline
		symmetry &$R^\perp$&$E^{(1)}_8$&$E^{(1)}_7$&$E^{(1)}_6$& $D^{(1)}_4$&$A^{(1)}_3$&$(A_1+A_1)^{(1)}$&$A^{(1)}_1$&$A^{(1)}_0$&$A^{(1)}_2$&$A^{(1)}_1$&$A^{(1)}_0$ \\ \hline
	\end{tabular}
\]
It is known that the difference Painlev\'e equations are obtained by discrete isomonodromic deformations of systems of linear differential equations (i.e. meromorphic connections on a trivial bundle over $\pl$). It implies that moduli spaces of suitable meromorphic connections can be regarded as the spaces of initial conditions for the difference Painlev\'e equations. Hence, a suitable family of moduli spaces of connections has an (extended) affine Weyl group symmetry. Then, the following problem naturally arises.

\begin{Problem}\label{problem}
Find transformations of meromorphic connections realizing the Weyl group symmetries.
\end{Problem}

This problem is theoretically answered in the case of linear differential equations by using the theory of quivers (for example, see \cite{BoaQu} and \cite{HirMo}).
Moduli spaces of linear differential equations with singularities are realized as quiver varieties (see Boalch \cite{BoaIr}, Crawley-Boevey \cite{Cr}, Hiroe and Yamakawa \cite{HY}, and Hiroe \cite{HirLi}). Quiver varieties have Weyl group symmetries generated by reflection functors. It is known that the reflection functors are translated to transformations of differential equations called addition and middle convolution. 

Our interest is in moduli spaces of meromorphic connections. The moduli space of linear differential equations is a proper open subset of the moduli space of meromorphic connections of degree zero. Hence, the theory of quivers is insufficient to explain symmetries of the whole moduli space of meromorphic connections. Our aim is to answer Problem \ref{problem} not using the theory of quivers, but only the theory of parabolic connections.


This paper is devoted to the case $R=A^{(1)*}_2$. The author \cite{Ma} showed that a suitable family of $A^{(1)*}_2$-surfaces is realized as a family of moduli spaces of rank three parabolic logarithmic connections on $\pl$ with three points (see $\S$\ref{logcon}). There is an explicit correspondence between points on $A^{(1)*}_2$-surfaces and parabolic logarithmic connections (see $\S$\ref{normform}). Through this correspondence, an action of the extended affine Weyl group $\widetilde{W}(E^{(1)}_6)$ is defined on the family of moduli spaces. We will explicitly write down transformations of parabolic connections realizing the isomorphisms between $A^{(1)*}_2$-surfaces. The most difficult realization is the reflection of the central node of the Dynkin diagram $E^{(1)}_6$. Many people believe that it is done by using middle convolutions. However, even when the linear differential equations are written down explicitly, no one knows how to express the transformations realizing the reflection.

We state the main theorem. Let $M(\boldsymbol{\nu})$ be the moduli space of $\boldsymbol{\nu}$-parabolic logarithmic connections and $S_{\boldsymbol{\nu}}$ be a suitable $A^{(1)*}_2$-surface. Let $w_3$ be the reflection of the central node of the Dynkin diagram $E^{(1)}_6$, and we denote the corresponding isomorphism $S_{\boldsymbol{\nu}}\rightarrow S_{w_3(\boldsymbol{\nu})}$ by $\phi_{w_3}$. We will see in Section \ref{symsec} that $\phi_{w_3}$ is induced by a quadratic transformation of the projective plane $\mbP^2$. Let $\varphi_{\boldsymbol{\nu}}\colon M(\boldsymbol{\nu})\to S_{\boldsymbol{\nu}}$ be the isomorphism in Theorem \ref{varphi}. We denote the middle convolution with meromorphic one form $\beta$ by $\mc_\beta$ (which will be explained in Section \ref{mcsec}).

\begin{theorem}[Theorem \ref{MT}]\label{MTintro}
	The following diagram commutes for a suitable $\beta$.
	\[
	\begin{tikzcd}
		M(\boldsymbol{\nu})\arrow[r, "\mc_\beta "]\arrow[d, "\varphi_{\boldsymbol{\nu}}"]&[40pt]M(w_3(\boldsymbol{\nu}))\arrow[d, "\varphi_{w_3(\boldsymbol{\nu})}"]\\
		S_{\boldsymbol{\nu}}\arrow[r, "\phi_{w_3}"]&S_{w_3(\boldsymbol{\nu})}
	\end{tikzcd}
	\]
\end{theorem}

We discuss the relation between Theorem \ref{MTintro} and the symmetries (B\"acklund transformations) of the difference Painlev\'e equations. It is well known that the Painlev\'e equations are obtained by isomonodromic deformations of linear differential equations. The commutativity of the isomonodromic deformation and middle convolution was established by Haraoka and Filipuk \cite{HF}, Boalch \cite{BoaSi}, and Yamakawa \cite{YamFou}. Hence, the symmetries of the Painlev\'e equations are explained by middle convolution. On the other hand, it is not known whether the difference Painlev\'e equations are invariant under middle convolution. Theorem \ref{MTintro} means that the middle convolution $\mc_{\beta}$ is compatible with the symmetries of spaces of initial conditions. This shows that the symmetries of the difference Painlev\'e equations of type $A^{(1)*}_2$ are explained by middle convolution.

We mention some related results. Inaba, Iwasaki, and Saito provided an explicit correspondence between a good family of $D^{(1)}_4$-surfaces and a family of moduli spaces of rank two parabolic logarithmic connections on $\pl$ with four points \cite{IIS2}. In this case, Haraoka and Filipuk realized the reflection of the central node by using middle convolution for linear differential equations \cite{HF}. We expect that our method also works well in the $D^{(1)}_4$-case. We note that Arinkin and Lysenko realized the reflection of the central node of the Dynkin diagram $D^{(1)}_4$ without middle convolution \cite{AL}. We don't know the relation of thier method and middle convolution. Yamakawa proved that Riemann--Hilbert map from moduli spaces of parabolic logarithmic connections to multiplicative quiver varieties is biholomorphic \cite{YamGeo}. Hence, the symmetry of moduli spaces of parabolic logarithmic connections is indirectly explained by the reflection functor of multiplicative quiver varieties.

This paper is organized as follows. Section 2 is devoted to the background results for this study. After making a family $\mcS\rightarrow \mcN$ of $A^{(1)*}_2$-surfaces, we construct the isomorphism from the family of moduli spaces of parabolic connections to the family of $A^{(1)*}_2$-surfaces. This construction is used to prove the main theorem. We also provide an explicit correspondence between points on $A^{(1)*}_2$-surfaces and parabolic connections. In Section 3, we explain how to define the action of $\widetilde{W}(E^{(1)}_6)$ on the family $\mcS\rightarrow \mcN$. After recalling the extended affine Weyl group $\widetilde{W}(E^{(1)}_6)$, we explain the action of $\widetilde{W}(E^{(1)}_6)$ on the Picard group of an $A^{(1)*}_2$-surface. This action leads to the actions on  $\mcN$ and $\mcS$. We summarize the explicit form of the action in the appendix. In Section 4, we realize the action on $\mcS$ as transformations of parabolic connections except for $w_3$, and state the main theorem. In Section 5, we review the theory of middle convolution by Simpson \cite{Sim}. In Section 6, we show the main theorem.

\section*{Acknowledgments}
The author is grateful to Koichi Takemura for suggesting the study of the realization of automorphism groups. He also thanks Shunya Adachi and Kazuki Hiroe for reading the introduction and providing useful comments. This work was supported by the Research Institute for Mathematical Sciences, an International Joint Usage/Research Center located in Kyoto University.

\section{$A^{(1)*}_2$-surfaces and moduli spaces of logarithmic connections}

\subsection{$A^{(1)*}_2$-surfaces}\label{a12surf}

Let $X$ be the surface obtained by blowing-up of the projective plane $\mathbb{P}^2$ at the following nine distinct points;
\[
\begin{tabular}{lll}
	$p_1: (1: -a_1 : 1)$,&$p_2: (1: -a_2 : 1)$, & $p_3: (1: -a_3 : 1)$, \\
	$p_4: (0: a_4 : 1)$,&$p_5: (0: a_5 : 1)$,& $p_9: (0: a_9 : 1)$, \\
	$p_6: (1: a_6 : 0)$,&$p_7: (1: a_7 : 0)$,&  $p_8: (1: a_8 : 0)$.
\end{tabular}
\]
Let $(x_0: x_1: x_2)$ be a homogeneous coordinate of $\mathbb{P}^2$ and set $\mcD_0:=\{x_0=0\}, \mcD_1:=\{x_0=x_2\}, \mcD_2:=\{x_2=0\}$. The cubic curve $\{x_0(x_0-x_2)x_2=0\}=\mcD_0\cup \mcD_1\cup \mcD_2$ passes through the above nine points.

\begin{lemma}
A cubic curve on $\mathbb{P}^2$ passing through $p_1, \ldots, p_9$ is only $\mcD_0\cup \mcD_1\cup\mcD_2$ if and only if $\sum_{i=1}^{9}a_i\neq0$.
\end{lemma}

We call $X$ \textit{an $A^{(1)*}_2$-surface} if an effective anti-canonical divisor of $X$ is unique. This condition is equivalent to $\sum_{i=1}^{9}a_i\neq 0$. 

We consider isomorphism classes of configurations of nine points.  We say that two collections of nine points of the above are isomorphic to each other if one moves the other by an isomorphism of $\mbP^2$ preserving each three lines $\mcD_0, \mcD_1, \mcD_2$. Such an isomorphism is of the form $(x_0: x_1: x_2) \mapsto (x_0: \mu x_0+\nu x_1+\eta x_2 : x_2)$. Let us fix complex numbers $\nu_0, \nu_1, \nu_\infty$ with $\nu_0+\nu_1+\nu_\infty=1$. We can see by calculation that each isomorphism class of configurations of nine points has a unique element with 
\begin{equation}\label{ptcond}
a_1+a_2+a_3=\nu_0, \ a_4+a_5+a_9=\nu_1, \ a_6+a_7+a_8=\nu_\infty. 
\end{equation}

We will construct a family of $A^{(1)*}_2$-surfaces obtained by blowing-up of $\mbP^2$ at $p_1, \ldots, p_9$ with the condition (\ref{ptcond}). Set
\[
\mcN:=\left\{\boldsymbol{\nu}=(\nu_{i,j})_{j=0, 1, 2}^{i=0, 1, \infty} \in \mathbb{C}^9 \vb 
\begin{array}{l}
\text{$\nu_{i,0}+\nu_{i,1}+\nu_{i,2}=\nu_i$ and $\nu_{i,0}\neq\nu_{i,1}\neq\nu_{i,2}\neq\nu_{i,0}$}\\
\text{for each $i=0, 1, \infty$}
\end{array}
\right\},
\]
where $\delta_{i, \infty}$ is the Kronecker delta.
Let $\mcS\rightarrow \mathbb{P}^2  \times \mcN$ be the blow-up along the nine sections $p_i \times \id \colon \mcN \rightarrow \mathbb{P}^2\times \mcN$;
\[
\begin{tabular}{lll}
	$p_1(\boldsymbol{\nu})=(1: \nu_{1,0}: 1)$,&$p_2(\boldsymbol{\nu})=(1: \nu_{1,1}: 1)$, & $p_3(\boldsymbol{\nu})=(1: \nu_{1,2}: 1)$, \\
	$p_4(\boldsymbol{\nu})=(0:-\nu_{0,0}: 1)$,&$p_5(\boldsymbol{\nu})=(0:-\nu_{0,1}: 1)$,& $p_9(\boldsymbol{\nu})=(0:-\nu_{0,2}: 1)$, \\
	$p_6(\boldsymbol{\nu})=(1: -\nu_{\infty,0}: 0)$,&$p_7(\boldsymbol{\nu})= (1: -\nu_{\infty,1}: 0)$,&  $p_8(\boldsymbol{\nu})=(1: -\nu_{\infty,2}: 0)$.
\end{tabular}
\]
The composition $\mcS\rightarrow \mathbb{P}^2  \times \mcN\rightarrow \mcN$ is then a family of $A^{(1)*}_2$-surfaces and the strict transform $\mcY$ of $\{x_0(x_0-x_2)x_2=0\} \times \mcN$ is a family of effective anti-canonical divisors. We denote the fiber of $\mcS$, $\mcY$ at $\boldsymbol{\nu}\in \mcN$ by $S_{\boldsymbol{\nu}}, Y_{\boldsymbol{\nu}}$, respectively. We will discuss the action of the extended affine Weyl group $\widetilde{W}(E^{(1)}_6)$ on $\mcS$ in Section \ref{symsec}.

\subsection{Moduli space of parabolic logarithmic connections}\label{logcon}
Take $n$ distinct points $t_1, \ldots, t_n \in \pl$ and set $D=[t_1]+\cdots+[t_n]$.
\begin{definition}
	 A logarithmic connection $(E,\nabla)$ of rank $r$ and degree $d$ over $(\pl, D)$ is a pair of a vector bundle $E$ of rank $r$ and degree $d$ on $\pl$ and a $\mathbb{C}$-homomorphism $\nabla \colon E\rightarrow E\otimes \Omega^1_{\pl}(D)$ satisfying Leibniz rule, i.e.,
	 \[
	 \nabla(fs)=s\otimes df+f\nabla(s)
	 \]
	 for any $f\in \Opl$ and $s\in E$. 
\end{definition}
Let $k(t_i)$ be the residue field.
For a logarithmic connection $(E, \nabla)$, let $\res_{t_i}(\nabla)$ be the endomorphism of $E|_{t_i}:=E\otimes k(t_i)$ induced by $\nabla$.
\begin{definition}\label{paraconn}
Take $\boldsymbol{\nu}=(\nu_{i,j})^{1\leq i\leq n}_{0\leq j \leq r-1} \in \mathbb{C}^{rn}$. A $\boldsymbol{\nu}$-parabolic logarithmic connection $(E,\nabla, l_*=\{l_{i,*}\}_{1 \leq i\leq n})$ of rank $r$ and degree $d$ over $(\pl, D)$ is a collection consisting of the following data;
\begin{itemize}\setlength{\itemsep}{0cm}
	\item[(1)] $(E, \nabla)$ is a logarithmic connection of rank $r$ and degree $d$ over $(\pl, D)$, 
	\item[(2)] $l_{i,*}$ is a filtration $E|_{t_i}:=E\otimes k(t_i)=l_{i,0} \supsetneq \cdots \supsetneq  l_{i,r-1} \supsetneq l_{i,r}=\{0\}$ by subspaces satisfying
	\[
	(\res_{t_i}(\nabla)-\nu_{i,j}\id )({l_{i,j}})\subset l_{i,j+1}
	\]
	for $1\leq i\leq n$ and $0\leq j\leq r-1$. We call $l_{i,*}$ a parabolic structure.
\end{itemize}
\end{definition}
\begin{remark}\label{remark}
(1)  For a $\boldsymbol{\nu}$-parabolic connection, $\nu_{i,j}$ is an eigenvalue of $\res_{t_i}(\nabla)$. We call $\nu_{i,j}$ is a local exponent of $\nabla$ at $t_i$. When $\nu_{i,j}\neq \nu_{i,k}$ for any $j\neq k$, a parabolic structure $l_{i,*}$ with (2) in Definition \ref{paraconn} is uniquely determined. In this paper, we consider the case $r=n=3$ and assume that $\nu_{i,0}, \nu_{i,1}, \nu_{i,2}$ are different. Then, parabolic structures may seem unnecessary since they are uniquely determined. The reason to consider parabolic structures is to show that the bijective maps of moduli spaces induced by transformations of connections are algebraic morphisms. 

(2)  A $\boldsymbol{\nu}$-parabolic logarithmic connection $(E,\nabla, l_*=\{l_{i,*}\}_{1 \leq i\leq n})$ satisfies the relation
\begin{equation}\label{Fuchs}
	\sum_{i=1}^n\sum_{j=0}^{r-1}\nu_{i,j}+\deg E=0.
\end{equation}
\end{remark}

We consider the case $n=3$ and $D=[0]+[1]+[\infty]$. Set
\[
\mcN^0:=\left\{\boldsymbol{\nu}=(\nu_{i,j})_{j=0, 1, 2}^{i=0, 1, \infty} \in \mcN \vb \text{$\nu_{0,j_0}+\nu_{1,j_1}+\nu_{\infty,j_\infty}\notin \mathbb{Z}$ for any $j_0, j_1, j_\infty$} \right\}, 
\]
where $\mcN$ is the set defined in the previous subsection. For $\boldsymbol{\nu} \in \mcN^{0}$, $\boldsymbol{\nu}$-parabolic connections $(E, \nabla)$ are irreducible, that is, there is no subbundle $F$ such that $\nabla(F)\subset F\otimes \Omega_{\pl}^1(D)$. In fact, if a subbundle $F\subset E$ is preserved by $\nabla$, then we can see that there are $j_0, j_1, j_\infty\in \{0, 1,2\}$ such that $\nu_{0,j_0}+\nu_{1,j_1}+\nu_{\infty,j_\infty}\in \mathbb{Z}$ by applying (\ref{Fuchs}) to the subconnection $(F, \nabla)$, which contradicts the condition of $\mcN^{0}$. Since any $\boldsymbol{\nu}$-parabolic connection is irreducible, we don't have to introduce stability conditions when considering moduli spaces.

Let $M(\boldsymbol{\nu})$ be the moduli space of $\boldsymbol{\nu}$-parabolic logarithmic connections and $\mcM\rightarrow \mcN^0$ be a family of moduli spaces of logarithmic connections whose fiber at $\boldsymbol{\nu}$ is $M(\boldsymbol{\nu})$. The existence of such moduli spaces and such a family is proved by Inaba \cite{In}. It is also shown that the moduli space is fine in his paper. In the case $(\nu_0, \nu_1, \nu_\infty)=(0, 0, 1)$, the following is a part of Theorem 3.1 in \cite{Ma}.

\begin{proposition}\label{varphi}
There is an isomorphism $\varphi: \mcM \longrightarrow \mcS \setminus \mcY$ over $\mcN^0$.
\[
\begin{tikzcd}
\mcM\arrow[rr, "\varphi"]\arrow[rd]&&\mcS\setminus \mcY\arrow[ld]\\
&\mcN^0&
\end{tikzcd}
\]
\end{proposition}
We can prove this by a similar argument of \cite{Ma}. We will explain it. 

First, we define the apparent singularity of $(E, \nabla,l_*)\in \mcM$. For $\boldsymbol{\nu} \in \mcN$, a $\boldsymbol{\nu}$-parabolic connection has degree $1$ by the relation (\ref{Fuchs}). When $\boldsymbol{\nu} \in \mcN^0$, the underlying bundle of a $\boldsymbol{\nu}$-logarithmic connection is isomorphic to $\mcO_{\pl}(1)\oplus \mcO_{\pl}^{\oplus 2}$ by the irreducibility of connections (for example, see Proposition 3.5 in \cite{Ma}).
\begin{proposition}\label{appfil}(Proposition 3.7 in \cite{Ma})
	For each $(E, \nabla,l_*)\in \mcM$,  there exists a unique filtration $E=E_0\supsetneq E_1\supsetneq E_2\supsetneq E_3=0$ by subbundles such that 
	\begin{equation}\label{filcon1}
		E_1\cong \Opl(1)\oplus \Opl, \  E_2\cong \Opl(1), \ \nabla(E_2)\subset E_1\otimes \Omega_{\pl}^1(D).
	\end{equation}
\end{proposition}
The homomorphism
\begin{equation}\label{apphom}
u: \Opl\cong E_1/E_2\overset{\nabla}{\rightarrow} (E/E_1)\otimes \Omega_{\pl}^1(D) \cong \Opl(1),
\end{equation}
is not zero since $(E, \nabla,l_*)$ is irreducible.
\begin{definition}
We call the zero of the homomorphism (\ref{apphom}) \textit{the apparent singularity} of $(E, \nabla,l_*)$. We denote it by $\App(E,\nabla,l_*)$.
\end{definition}
The apparent singularity defines a morphism $\App \colon \mcM\rightarrow \mbP^1$. 

Next, we construct a morphism $ \mcM\rightarrow \mathbb{P}(\Omega_{\pl}^1(D) \oplus \Opl)$. Take $(E, \nabla, l_*)\in \mcM$ and put $q:=\App(E, \nabla, l_*)$. 
Let $\pi \colon E \rightarrow E/E_1$ be the quotient and let us fix an isomorphism $E/E_1\cong \Opl(-t_3)$. We define a homomorphism $B \colon E \rightarrow  E/E_1 \otimes \Omega_{\pl}^1(D)$ by $B(a):=(\pi \otimes \id) \nabla(a)-d(\pi(a))$ for $a \in E$, where $d$ is the canonical connection on $\Opl(-t_3)$. Since $\nabla(E_2) \subset E_1 \otimes \Omega_{\pl}^1(D)$ and $u_q=0$, $B_q$ induces a homomorphism $h \colon (E/E_1)|_q \rightarrow (E/E_1\otimes\Omega_{\pl}^1(D))|_q$ which makes the diagram 
\begin{equation}\label{h1diag}
	\begin{tikzcd}
		0 \arrow[r]&E_1|_q\arrow[r]\arrow[rd,"0"]&E|_q \arrow[r]\arrow[d, "B_q"]&(E/E_1)|_q\arrow[r] \arrow[ld, "h"]&0 \\
		&&(E/E_1\otimes \Omega_{\pl}^1(D))|_q&&
	\end{tikzcd}
\end{equation}
commute. Then $h$ determines a homomorphism 
\[
\iota \colon (E/E_1)|_q \longrightarrow ((E/E_1\otimes\Ompld)\oplus E/E_1)|_q, \quad a \mapsto (h(a), a).
\]
Since $\iota$ is injective, the map $\iota$ determines a point $\phi(E, \nabla, l_*)$ of $\mathbb{P}(\Ompld \oplus \Opl)$. We can see that the map 
\begin{equation}\label{phidef}
	\phi \colon \mcM\longrightarrow \mathbb{P}(\Omega_{\pl}^1(D) \oplus \Opl)
\end{equation}
is a morphism. 

Next, we construct a morphism $ \mcM\rightarrow \mathbb{P}^2$.
We find that the image of $\phi$ and the infinite section of $\mathbb{P}(\Omega_{\pl}^1(D) \oplus \Opl)$ don't intersect by the construction of $\phi$. Since $\Omega_{\pl}^1(D) \cong \Opl(1)$, we have a morphism $\mathbb{P}(\Omega_{\pl}^1(D) \oplus \Opl) \rightarrow \mbP^2$ by contructing the infinite section.  We hence get a morphism $\mcM \rightarrow  \mbP^2$ by composing the morphism $\phi$ with the construction. 

Therefore, we obtain a morphism  $\mcM\rightarrow \mbP^2\times \mcN^0$. We can see that this lifts to a morphism $\mcM\rightarrow \mcS\setminus \mcY$. We can show Proposition \ref{varphi} by studying the correspondence between points on $M(\boldsymbol{\nu})$ and $S_{\boldsymbol{\nu}}\setminus Y_{\boldsymbol{\nu}}$ for all $\boldsymbol{\nu} \in \mcN^0$. We will see it in the next section.

\subsection{Correspondence between points on the $A^{(1)*}_2$-surface and  parabolic connections}\label{normform}
We provide an explicit correspondence to the $A^{(1)*}_2$-surface $S_{\boldsymbol{\nu}}$ and the moduli space $M(\boldsymbol{\nu})$ when $\boldsymbol{\nu} \in \mcN^0$. The underlying bundle of a $\boldsymbol{\nu}$-parabolic connection is of the form $\mcO_{\pl}(1)\oplus \mcO_{\pl}^{\oplus 2}$, and its parabolic structure is uniquely determined by definition. We give the normal form of connections.

Let $(c_0:c_1)$ be a homogeneous coordinate of $\pl$ and let $U_0:=\{c_0\neq 0\}\subset \pl, \ U_\infty:=\{c_1\neq 0\} \subset \pl, \ z:=c_1/c
_0, \ w:=c_0/c_1.$
We take a local basis $e^{(0)}_0, e^{(0)}_1, e^{(0)}_2$ (resp. $e^{(\infty)}_0, e^{(\infty)}_1, e^{(\infty)}_2$) of $\Opl\oplus\Opl(-1)\oplus\Opl(-1)$ on $U_0$ (resp. on $U_\infty$) satisfying 
$e^{(0)}_0=e^{(\infty)}_0$, $e^{(0)}_1=\frac{1}{w}e^{(\infty)}_1$,$e^{(0)}_2=\frac{1}{w}e^{(\infty)}_2$. For simplicity of notation, we write $\nabla=d+A$ on $U_0$ (resp. on $U_\infty$) instead of $\nabla(e^{(0)}_0, e^{(0)}_1, e^{(0)}_2)=(e^{(0)}_0, e^{(0)}_1, e^{(0)}_2)A$ (resp. $\nabla(e^{(\infty)}_0, e^{(\infty)}_1, e^{(\infty)}_2)=(e^{(\infty)}_0, e^{(\infty)}_1, e^{(\infty)}_2)A$). The correspondence is as follows;
\begin{equation}\label{normformeq}
(q:p:1)\longleftrightarrow \nabla=d+
\begin{pmatrix}
	0& a_{12}(z)& a_{13}(z)\\
	1& (\nu_0+\nu_1)z-p-\nu_0& 0\\
	0& z-q& p
\end{pmatrix}
\frac{dz}{z(z-1)}\quad \text{on $U_0$},
\end{equation}
\[
(1:p':q')\longleftrightarrow 
\nabla=d+
\begin{pmatrix}
0& b_{12}(w)& b_{13}(w)\\
1& (\nu_1+\nu_2)w-p'-\nu_2& 0\\
0& w-q'& p'
\end{pmatrix}
\frac{dw}{w(w-1)} \quad \text{on $U_\infty$}.
\]
Here $a_{12}(z), a_{13}(z), b_{12}(w), b_{13}(w)$ are the quadratic polynomials satisfying
\begin{align*}
a_{12}(0)=&-p^2-\nu_0p-(\nu_{0,0}\nu_{0,1}+\nu_{0,1}\nu_{0,2}+\nu_{0,2}\nu_{0,0}),\\
a_{12}(1)=&-p^2+\nu_0p-(\nu_{1,0}\nu_{1,1}+\nu_{1,1}\nu_{1,2}+\nu_{1,2}\nu_{1,0}),\\
\lim_{z\rightarrow \infty}a_{12}(z)/z^2=&-(1+\nu_2+\nu_{\infty,0}\nu_{\infty,1}+\nu_{\infty,1}\nu_{\infty,2}+\nu_{\infty,2}\nu_{\infty,0}),\\
qa_{13}(0)=&(p+\nu_{0,0})(p+\nu_{0,1})(p+\nu_{0,2}),\\
(q-1)a_{13}(1)=&(p-\nu_{1,0})(p-\nu_{1,1})(p-\nu_{1,2}),\\
\lim_{z\rightarrow \infty}a_{13}(z)/z^2=&-\nu_{\infty,0}\nu_{\infty,1}\nu_{\infty,2},\\
b_{12}(0)=&-(p')^2-\nu_2p'-(\nu_{\infty,0}\nu_{\infty,1}+\nu_{\infty,1}\nu_{\infty,2}+\nu_{\infty,2}\nu_{\infty,0}),\\
b_{12}(1)=&-p'^2+\nu_1p'-(\nu_{1,0}\nu_{1,1}+\nu_{1,1}\nu_{1,2}+\nu_{1,2}\nu_{1,0}),\\
\lim_{w\rightarrow \infty}b_{12}(w)/w^2=&-(1+\nu_0+\nu_{0,0}\nu_{0,1}+\nu_{0,1}\nu_{0,2}+\nu_{0,2}\nu_{0,0}),\\
q'b_{13}(0)=&(p'+\nu_{\infty,0})(p'+\nu_{\infty,1})(p'+\nu_{\infty,2}),\\
(q'-1)b_{13}(1)=&(p'-\nu_{1,0})(p'-\nu_{1,1})(p'-\nu_{1,2}),\\
\lim_{w\rightarrow \infty}b_{13}(w)/w^2=&-\nu_{0,0}\nu_{0,1}\nu_{0,2}.
\end{align*}
We can see that the parabolic connections corresponding to $(q:p:1)$ and $(1:p':q')$ are isomorphic to each other if and only if $q'=q^{-1}$ and $p'=pq^{-1}$. We note that the exceptional divisors of $S_{\boldsymbol{\nu}}$ comes from the ambiguity of $a{13}(0), a{13}(1), b{13}(0)$.

Let $E_0\supsetneq E_1\supsetneq E_2\supsetneq E_3$ be the filtration in Proposition \ref{appfil}. We can find by the form of $\nabla$ that $E_2$ is generated by $e^{(0)}_0$, and  $E_1$ is generated by $e^{(0)}_0, e^{(0)}_1$. The homomorphism $u$ in (\ref{apphom}) is given by $u(e^{(0)}_1)=(z-q)e^{(0)}_2$. Thus $q$ is the apparent singularity. The homomorphism $h$ in (\ref{phidef}) is given by $h(e^{(0)}_2)=e^{(0)}_2\otimes p\tfrac{dz}{z(z-1)}$. So $p$ is characterized as the coefficient of $\frac{dz}{z(z-1)}$. We call $p$ the \textit{dual parameter}. We use this characterization in the proof of the main theorem.

\section{Symmetry of $A^{(1)*}_2$-surfaces}\label{symsec}

\subsection{The extended affine Weyl group $\widetilde{W}(E^{(1)}_6)$}\label{E6}
The extended affine Weyl group $\widetilde{W}(E^{(1)}_6)$ is the semi-product  $\Aut(E^{(1)}_6) \ltimes W(E^{(1)}_6)$. The affine Weyl group $W(E^{(1)}_6)$ is defined as follows;
\[
W(E^{(1)}_6):=\left\langle  w_0, \ldots, w_6 \vb
\begin{array}{cc}
	w_i^2=1& i=0, \cdots, 6\\
	w_iw_j=w_jw_i&\text{$i$ and $j$ are not joined}\\
	w_iw_jw_i=w_jw_iw_j&\text{$i$ and $j$ are joined}
\end{array}
\right\rangle. 
\]
Here $i$ and $j$ mean the nodes of the Dynkin diagram of type $E^{(1)}_6$ in (\ref{aute}).
The automorphism group $\Aut(E^{(1)}_6)$ of the Dynkin diagram of type $E^{(1)}_6$  is generated by $\sigma_{1}$ and $\sigma_{2}$ defined as follows.
\begin{equation}\label{aute}
\centering
\includegraphics[width=7.4cm]{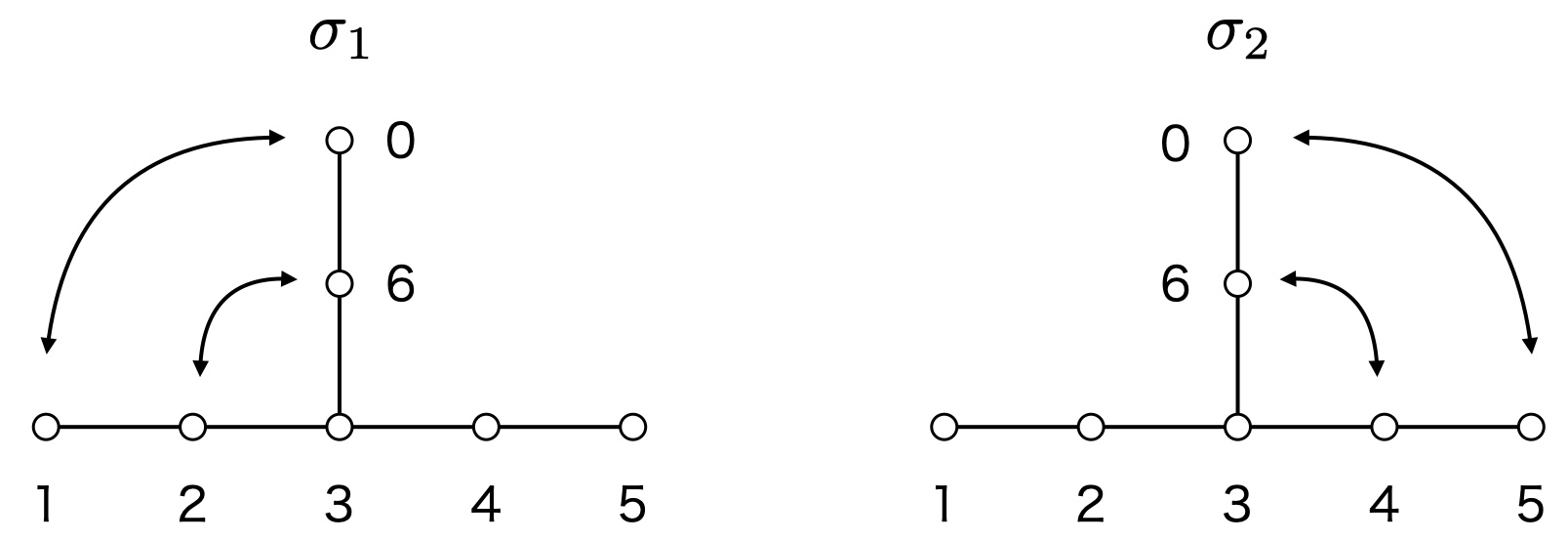}
\end{equation}
We note that $\Aut(E^{(1)}_6)$ is also the autmorphism group of the Dynkin diagram of type $A^{(1)}_2$. We define automorphisms corresponding to $\sigma_{1}$ and $\sigma_{2}$ as follows.
\begin{equation}\label{auta}
\centering
\includegraphics[width=46mm]{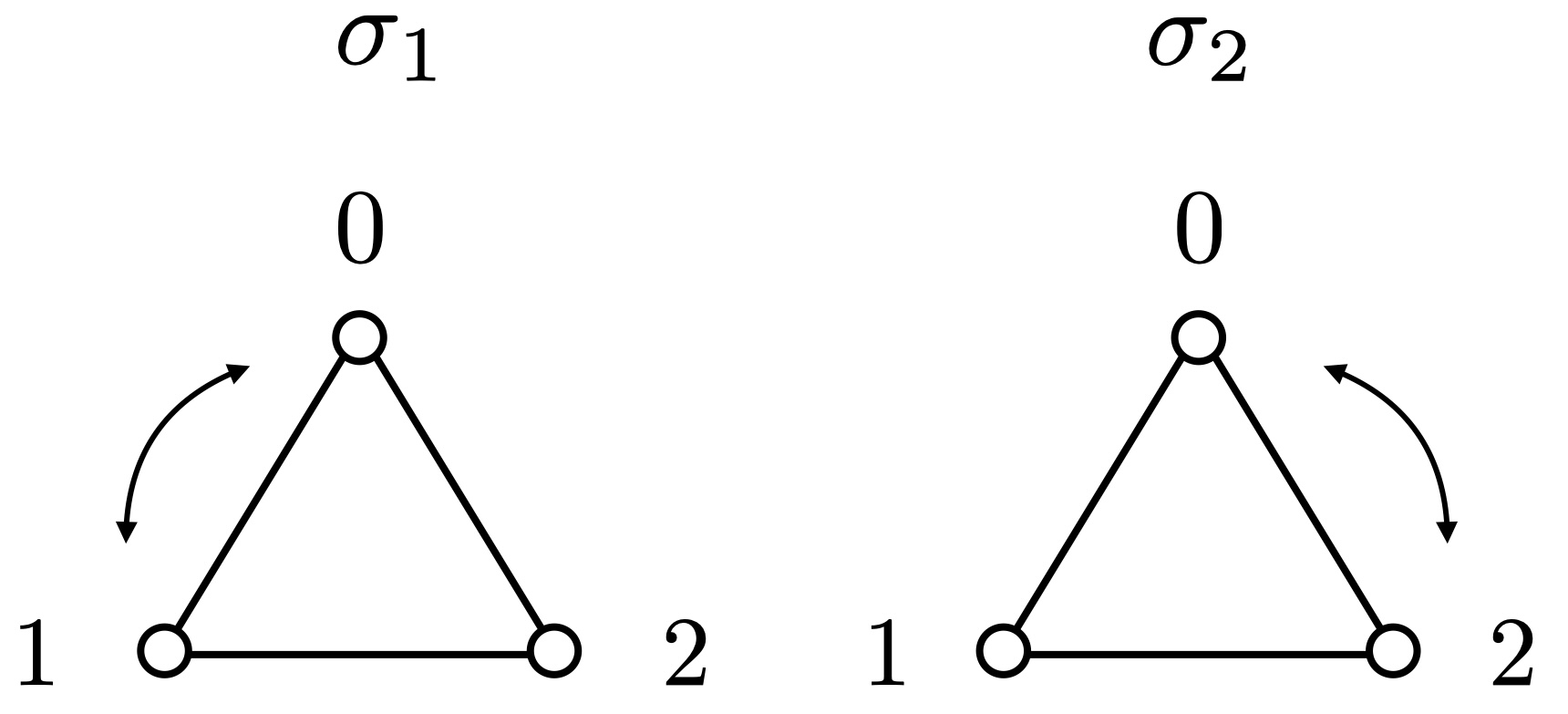}
\end{equation}

\subsection{Action of $\widetilde{W}(E^{(1)}_6)$ on $\Pic (S_{\boldsymbol{\nu}})$}

The surface $S_{\boldsymbol{\nu}}$ is obtained by blowing up $\mbP^2$ at the nine points $p_1(\boldsymbol{\nu}), \ldots, p_9(\boldsymbol{\nu})$ (see section \ref{a12surf}). 
Let $\mcE_0$ be a line on $\mbP^2$ and $\mcE_i$ be the exceptional divisor of $S_{\boldsymbol{\nu}}$ at $p_i(\boldsymbol{\nu})$ for $i=1, \ldots, 9$. The Picard group $\Pic (S_{\boldsymbol{\nu}})$ then has the decomposition $\Pic (S_{\boldsymbol{\nu}})=\bigoplus_{i=0}^9\mbZ \mcE_i$. 

First, we consider the action of $W(E^{(1)}_6)$ on $\Pic (S_{\boldsymbol{\nu}})$. Set

\begin{equation}
	\begin{array}{lll}
		\alpha_1=\mcE_2-\mcE_3, & \alpha_2=\mcE_1-\mcE_2, &\alpha_3=\mcE_0-\mcE_1-\mcE_4-\mcE_6, \\
		\alpha_4=\mcE_6-\mcE_7, & \alpha_5=\mcE_7-\mcE_8, &\alpha_6=\mcE_4-\mcE_5, \quad \alpha_0=\mcE_5-\mcE_9.
	\end{array}
\end{equation}
The sublattice generated by $\alpha_0, \cdots, \alpha_6$ is isomorphic to the root lattice of type $E^{(1)}_6$. We denote it by $Q(E^{(1)}_6)$, i.e., $Q(E^{(1)}_6):=\sum_{i=0}^6\mbZ \alpha_i$. The affine Weyl group $W(E^{(1)}_6)$ is isomorphic to the group generated by the reflections of $\alpha_0, \cdots, \alpha_6$ on $Q(E^{(1)}_6)$. The action of $W(E^{(1)}_6)$ on $Q(E^{(1)}_6)$ is naturally extended on $\Pic (S_{\boldsymbol{\nu}})$, that is, $W(E^{(1)}_6)$ acts on $\Pic (S_{\boldsymbol{\nu}})$ by $w_i(\mcD)=\mcD+(\mcD\cdot \alpha_i)\alpha_i$.

Next, we consider the action of $\Aut(E^{(1)}_6)$ on $\Pic (S_{\boldsymbol{\nu}})$. We denote the strict transform of the line $\mcD_i$ in Section \ref{a12surf} by the same character. $D_0, D_1, D_2$ are written as elements in $\Pic (S_{\boldsymbol{\nu}})$ as follows;
\[
\mcD_{0}=\mcE_0-\mcE_4-\mcE_5-\mcE_9, \quad \mcD_{1}=\mcE_0-\mcE_1-\mcE_2-\mcE_3, \quad \mcD_{2}=\mcE_0-\mcE_6-\mcE_7-\mcE_8.
\]
The sublattice generated by $\mcD_0, \mcD_1, \mcD_2$ is isomorphic the root lattice of type $A^{(1)}_2$. We denote it by $Q(A^{(1)}_2)$, i.e., $Q(A^{(1)}_2):=\sum_{i=0}^2\mbZ \mcD_i$. We can see that $Q(E^{(1)}_6)+Q(A^{(1)}_2)=\Pic (S_{\boldsymbol{\nu}})$ and $Q(E^{(1)}_6) \cap Q(A^{(1)}_2) =\mbZ \delta$, where
\[
\delta=\alpha_1+2\alpha_2+3\alpha_3+2\alpha_4+\alpha_5+2\alpha_6+\alpha_0=\mcD_{0}+\mcD_{1}+\mcD_{2}.
\]
We set $\sigma_i(\alpha_j)=\alpha_{\sigma_i(j)}$ (resp. $\sigma_i(\mcD_k)=\mcD_{\sigma_i(k)}$) by the figure (\ref{aute}) (resp. (\ref{auta})). Then $\sigma_i$ gives automorphisms of $Q(E^{(1)}_6)$ and $Q(A^{(1)}_2)$. This action extends to $Q(E^{(1)}_6)+Q(A^{(1)}_2)=\Pic (S_{\boldsymbol{\nu}})$, because it defines the same action on $Q(E^{(1)}_6) \cap Q(A^{(1)}_2)$. In fact, we can check
\[
\sigma_i(\alpha_1)+2\sigma_i(\alpha_2)+3\sigma_i(\alpha_3)+2\sigma_i(\alpha_4)+\sigma_i(\alpha_5)+2\sigma_i(\alpha_6)+\sigma_i(\alpha_7)=\sigma_i(\mcD_0)+\sigma_i(\mcD_1)+\sigma_i(\mcD_2)
\]
for $i=1,2$. Therefore $\sigma_i$ provides an autmorphism of $\Pic (S_{\boldsymbol{\nu}})$. We find that these automorphisms define an action of $\Aut(E^{(1)}_6)$ on $\Pic (S_{\boldsymbol{\nu}})$ by the construction. 

In the appendix, we give variations of the basis $\mcE_0, \ldots, \mcE_9$ by generators of $\widetilde{W}(E^{(1)}_6)$. 

\subsection{Action of $\widetilde{W}(E^{(1)}_6)$ on $\mcN^0$}

The action of $\widetilde{W}(E^{(1)}_6)$ on $\mcN^0$ comes from period integrals. We consider a meromorphic two-form on $S_{\boldsymbol{\nu}}$ with a simple pole along the effective anti-canonical divisor $Y_{\boldsymbol{\nu}}$. Since $H^0(S_{\boldsymbol{\nu}}, \Omega_{S_{\boldsymbol{\nu}}}^2(Y_{\boldsymbol{\nu}})) \cong H^0(S_{\boldsymbol{\nu}}, \mcO_{S_{\boldsymbol{\nu}}})\cong \mathbb{C}$,  there is a unique meromorphic two-form $\omega$ in  $H^0(S_{\boldsymbol{\nu}}, \Omega_{S_{\boldsymbol{\nu}}}^2(Y_{\boldsymbol{\nu}}))$ with $\int_{\delta} \omega=1$. We define $\chi\colon Q(E^{(1)}_6)\rightarrow \mathbb{C}$ by $\chi(\mcD)=\int_{\mcD}\omega$. By Lemma 21 in \cite{SaRa}, we have
\begin{equation}\label{int}
\begin{array}{lll}
	\chi(\alpha_1)=\nu_{1,1}-\nu_{1,2},  &\chi(\alpha_2)=\nu_{1,0}-\nu_{1,1},  &\chi(\alpha_3)=\nu_{0,0}+\nu_{1,0}+\nu_{\infty,0}\\
	\chi(\alpha_4)=-\nu_{0,1}+\nu_{0,2},  &\chi(\alpha_5)=-\nu_{0,1}+\nu_{0,2},  &\chi(\alpha_6)=-\nu_{\infty,1}+\nu_{\infty,2}, \\
	\chi(\alpha_0)= -\nu_{\infty,1}+\nu_{\infty,2}.&&
\end{array}	
\end{equation}
In particular, $(\nu_{i,j})^{i=0, 1, \infty}_{j=0, 1, 2}$ is uniquely determined by the values $\chi(\alpha_0), \ldots, \chi(\alpha_9)$ and the condition
\begin{equation}\label{sumcond}
\nu_{i,0}+\nu_{i,1}+\nu_{i,2}=\nu_i
\end{equation}
for $i=0, 1, \infty$. We define an action of $\widetilde{W}(E^{(1)}_6)$ on $\mcN^0$ by using this result, that is, we define $(w(\nu_{i,j}))^{i=0, 1, \infty}_{j=0, 1, 2}$ as the solution of systems of linear equations obtained by replacing $\alpha_i$ and $\nu_{j,k}$ with $w(\alpha_i)$ and $w(\nu_{j,k})$ in the conditions (\ref{int}) and (\ref{sumcond}). 

We consider the case $w=w_1$ as an example. Focusing on the first equation of (\ref{int}), we have
\[
w_1(\nu_{1,1})-w_1(\nu_{1,2})=\chi(w_1(\alpha_1))=\chi(-\alpha_1)=-\nu_{1,1}+\nu_{1,2}.
\]
The following equation is obtained by repeating the same calculation. 
\[\left\{
\begin{array}{ll}
	w_1(\nu_{1,1})-w_1(\nu_{1,2})=-\nu_{1,1}+\nu_{1,2}& \\
	w_1(\nu_{1,0})-w_1(\nu_{1,1})=\nu_{1,0}-\nu_{1,2}& \\
	w_1(\nu_{i,j})-w_1(\nu_{i,j+1})=\nu_{i,j}-\nu_{i,j+1}& i=0, \infty, \ j=0, 1\\
	w_1(\nu_{i,0})+w_1(\nu_{i,1})+w_1(\nu_{i,2})=\nu_i &i=0, 1, \infty
\end{array}\right.
\]
The solution is as follows;
\begin{equation}\label{w1n}
w_1(\nu_{i,j})=
\left\{
\begin{array}{cc}
\nu_{i,(12)(j)} & i=1\\
\nu_{i,j} & i\neq 1.
\end{array}
\right.
\end{equation}

We give the explicit action of the generators of $\widetilde{W}(E^{(1)}_6)$ in the appendix.

\subsection{Action of $\widetilde{W}(E^{(1)}_6)$ on $\mcS$}
The action of $\widetilde{W}(E^{(1)}_6)$ on $\mcS$ is defined by the automorphisms which realize the action on the Picard group.

Let us fix $w \in \widetilde{W}(E^{(1)}_6)$. Let $\mcE'_0$ be a line on $\mbP^2$ and $\mcE'_i$ be the exceptional divisor of $S_{w(\boldsymbol{\nu})}$ at $p_i(w(\boldsymbol{\nu}))$ for $i=1, \ldots, 9$. We then have $\Pic (S_{w(\boldsymbol{\nu})})=\bigoplus_{i=0}^9\mbZ \mcE'_i$.
\begin{proposition}\label{actsurf}
There is a unique isomorphism $\phi_w \colon S_{\boldsymbol{\nu}}\rightarrow S_{w(\boldsymbol{\nu})}$ such that
\begin{equation}\label{actS}
\phi_w^*(\mcE'_i)=w(\mcE_i)
\end{equation}
for any $i =0, \cdots, 9$.
\end{proposition}

\begin{proof}
The uniqueness is followed by Theorem 25 in \cite{SaRa}. We prove the existence for each generator of $\widetilde{W}(E^{(1)}_6)$.

First, we consider the case $w=w_i$ for $i \neq 3$. The action of $w_1$ on $\Pic(S_{\boldsymbol{\nu}})$ is given by
\[
w_1(\mcE_i)=\mcE_{(23)(i)}.
\]
Since the sets of points $\{p_i(\boldsymbol{\nu})\}_{i=1}^9$ and $\{p_i(w_1(\boldsymbol{\nu}))\}_{i=1}^9$ are the same by (\ref{w1n}), the surfaces $S_{\boldsymbol{\nu}}$ and $S_{w_1(\boldsymbol{\nu})}$ are also the same. We can check that 
$\phi_{w_1}:=\id \colon S_{\boldsymbol{\nu}}\rightarrow S_{w_1(\boldsymbol{\nu})}$ satisfies the condition (\ref{actS}). We can also see that $\phi_{w_i}:=\id \colon S_{\boldsymbol{\nu}}\rightarrow S_{w_i(\boldsymbol{\nu})}$ is the desired isomorphism for other cases.

Next, we consider the case $w=\sigma_i$. The action of $\sigma_1$ on $\Pic(S_{\boldsymbol{\nu}})$ is given by
\[
\sigma_{1}(\mcE_i)=\mcE_{(14)(25)(39)(i)},
\]
and we have $\sigma_1(\mcD_j)=\mcD_{(01)(j)}$. It means that $\phi_{\sigma_1}$ should exchange lines $\mcD_0, \mcD_1$ and moves the exceptional curve $\mcE_i$ to $\mcE_{(14)(25)(39)(i)}$. 
Let $\phi_{\sigma_1}\colon S_{\boldsymbol{\nu}} \rightarrow S_{\sigma_1(\boldsymbol{\nu})}$ be the isomorphisms induced by the isomorphism
\[
\mathbb{P}^2 \rightarrow \mathbb{P}^2; \  (x_0:x_1:x_2)\mapsto (x_2-x_0: -\tfrac{1}{3}(\nu_0-\nu_1)x_2-x_1: x_2).
\]
We can see that $\phi_{\sigma_1}$ satisfies the condition (\ref{actS}).

In the case $w=\sigma_2$, we can find that the isomorphism $\phi_{\sigma_2}\colon S_{\boldsymbol{\nu}} \rightarrow S_{\sigma_2(\boldsymbol{\nu})}$ induced by 
\[
\mathbb{P}^2 \rightarrow \mathbb{P}^2; \  (x_0:x_1:x_2)\mapsto (x_2: -\tfrac{1}{3}(\nu_0-\nu_2)(x_0-x_2)+x_1: x_0)
\]
satisfies the condition (\ref{actS}).

Finally, we consider the case $w=w_3$.  The action of $w_3$ on $\Pic(S_{\boldsymbol{\nu}})$ is given by
\[
w_3(\mcE_j)=
\left\{
\begin{array}{cc}
	\mcE_0-\mcE_4-\mcE_6 & j=1\\
	\mcE_0-\mcE_1-\mcE_6 & j=4\\
	\mcE_0-\mcE_1-\mcE_4 & j=6\\
	\mcE_j & j\neq 1, 4, 6.
\end{array}
\right.
\]
It implies that $\phi_{w_3}$ is the morphism induced by the quadratic transformation of $\mathbb{P}^2$ based at the points $p_1(\boldsymbol{\nu}), p_4(\boldsymbol{\nu}), p_6(\boldsymbol{\nu})$.
Let  $\phi_{w_3} \colon S_{\boldsymbol{\nu}}\rightarrow S_{w_3(\boldsymbol{\nu})}$ be the isomorphsm induced by the quadratic transformation
\begin{equation}\label{quadtranf}
\mathbb{P}^2 \dasharrow \mathbb{P}^2, \  (x_0:x_1:x_2)\mapsto (\gamma x_0f_{16}: af_{14}f_{16}-bf_{16}f_{46}-cf_{14}f_{46} :\gamma x_2f_{14}), 
\end{equation}

where
\[
\gamma:=\nu_{0,0}+\nu_{1,0}+\nu_{\infty,0}, \ 
a:=\nu_{1,0}-\tfrac{2}{3}\gamma, \ b:=\nu_{\infty,0}-\tfrac{2}{3}\gamma, \  c:=\nu_{0,0}-\tfrac{2}{3}\gamma,
\]
\begin{align*}
&f_{14}(x_0, x_1, x_2):= \nu_{1,0}x_0+\nu_{0,0}(x_0-x_2)-x_1,\\
&f_{16}(x_0, x_1, x_2):=  \nu_{1,0}x_2+\nu_{\infty,0}(x_2-x_0)-x_1,\\
&f_{46}(x_0, x_1, x_2):= \nu_{0,0}x_2+\nu_{\infty,0}x_0+x_1.
\end{align*}
We can check that $\phi_{w_3}$ satisfies the condition (\ref{actS}).
\end{proof}
Isomorphisms in Proposition \ref{actsurf} are extended to automorphisms of $\mcS$. Those automorphisms define an action of $\widetilde{W}(E^{(1)}_6)$ on $\mcS$. The action is equivariant for the map $\mcS\rightarrow \mcN^0$ by the construction. In the next section, we will see realizations of the action over $\mcS\rightarrow \mcN^0$ in terms of parabolic connections.

\section{Symmetry of moduli spaces of parabolic connections}

Transformations of parabolic connections induce automorphisms of $\mcM$ by the universal property of the moduli spaces.
Our goal is to find transformations for any $w\in \widetilde{W}(E^{(1)}_6)$ such that the induced morphism $f_w\colon \mcM\rightarrow \mcM$ makes the following diagrams commutative
\begin{equation}\label{diag}
\begin{tikzcd}
	\mcM\arrow[r, "f_w"]\arrow[d]&\mcM\arrow[d]\\
	\mcN^0\arrow[r, "w"]&\mcN^0,
\end{tikzcd}  \quad 
\begin{tikzcd}
	\mcM\arrow[r, "f_w"]\arrow[d, "\varphi"]&\mcM\arrow[d, "\varphi"]\\
	\mcS\arrow[r, "\phi_w"]&\mcS.
\end{tikzcd}
\end{equation}
This is accomplished by thinking about the generators of $\widetilde{W}(E^{(1)}_6)$. We will see such a realization in three cases;
\[
\text{(i) $w_i\ (i\neq 3)$}, \quad \quad    \text{(ii)}\  \sigma_1, \sigma_2, \quad \quad \text{(iii)} \ w_3.
\]

\subsection{Action of $w_i$ on $\mcM$ when $i\neq 3$}
The action of $w_i$ on $\mcN^0$ is to exchange adjacent local exponents when $i\neq 3$. The action of $w_i$ on $\mcM$ is hence realized by changing not the underlying logarithmic connection but the parabolic structures.
We will see it in the cases of $w_1$ and $w_2$. We can realize the action of another $w_i$ in the same manner. $w_1$ acts on $\mcN^0$ by
\[
w_1(\nu_{i,j})=
\left\{
\begin{array}{cc}
	\nu_{i,(12)(j)} & i=1\\
	\nu_{i,j} & i\neq 1,
\end{array}
\right. 
\]
So we only have to replace $l_{1,2}$ with the eigenspace of $\res_{1}\nabla$ for $\nu_{1,1}$.
That is, for a $\boldsymbol{\nu}$-parabolic logarithmic connection $(E,\nabla, l_*)$, we define a $w_1(\boldsymbol{\nu})$-parabolic logarithmic connection $(E',\nabla', l'_*)$ by $E':=E$, $\nabla':=\nabla$, $l'_{i,j}:=l_{i,j}$ for $(i,j)\neq(1,2)$, and 
\[
l'_{1,2}:=\Ker(\res_1\nabla-\nu_{1,1}\id\colon E|_{1}\longrightarrow E|_{1})
\]
The transformation $(E,\nabla, l_*) \mapsto (E',\nabla', l'_*)$ induces an automorphism $f_{w_1}\colon \mcM\rightarrow \mcM
$, and $f_{w_1}$ commutes the diagrams (\ref{diag}). Next we consider $w_2$, which acts on $\mcN^0$ by
\[
w_2(\nu_{i,j})=
\left\{
\begin{array}{cc}
	\nu_{i,(01)(j)} & i=1\\
	\nu_{i,j} & i\neq 1,
\end{array}
\right. 
\]
We have to replace $l_{1,1}$ with the sum of the eigenspaces of $\res_{1}\nabla$ for $\nu_{1,0}$ and $\nu_{1,2}$.
For a $\boldsymbol{\nu}$-parabolic logarithmic connection $(E,\nabla, l_*)$, we define a $w_2(\boldsymbol{\nu})$-parabolic logarithmic connection $(E'',\nabla'', l''_*)$ by $E'':=E$, $\nabla'':=\nabla$, $l''_{i,j}:=l_{i,j}$ for $(i,j)\neq(1,1)$, and 
\[
l''_{1,1}:=\Ker(E|_{1}\longrightarrow E|_{1}/l_{1,2}\overset{\res_1\nabla-\nu_{1,0}\id}{\longrightarrow} E|_{1}/l_{1,2}).
\]
This transformation $(E,\nabla, l_*) \mapsto (E'',\nabla'', l''_*)$ is a desired one.

\subsection{Action of $\sigma_{i}$ on $\mcM$}
First, we consider $\sigma_1$.  The action of $\sigma_1$ on $\mcN^0$ is as follows;
\[
\sigma_1(\nu_{i,j})=
\left\{
\begin{array}{cc}
	\nu_{1, j}+\frac{1}{3}(\nu_0-\nu_1) & i=0\\
	\nu_{0, j}+\frac{1}{3}(\nu_1-\nu_0) & i=1\\
	\nu_{\infty, j} & i=\infty,
\end{array}
\right. 
\]
We define a morphism  $s_1\colon \pl \rightarrow \pl$ by $z\mapsto 1-z$. 
\begin{proposition}
Given a $\boldsymbol{\nu}$-logarithmic connection $(E, \nabla, l_*)$, set 
\[
(E', \nabla'):=s_1^*(E, \nabla)\otimes (\mcO_{\pl}, d-\tfrac{\nu_0-\nu_1}{3z(z-1)}dz).
\]
Then, there is a unique parabolic structure $l'_{i,*}$ of $E'|_{i}$ for each $i$ such that the collection $(E', \nabla', l'_*:=\{l'_{i,*}\}_{i=0, 1, \infty})$ is a $\sigma_{1}(\boldsymbol{\nu})$-parabolic connection. Let us define a morphism $f_{\sigma_1}\colon \mcM\rightarrow \mcM$ by $f_{\sigma_1}((E, \nabla, l_*)):=(E', \nabla', l'_*)$. Then $f_{\sigma_1}$ commutes the diagrams (\ref{diag}). 
\end{proposition}
\begin{proof}
We can easily check that eigenvalues of $\res_{i}\nabla'$ are $\sigma_{1}(\boldsymbol{\nu})_{i,1}, \sigma_{1}(\boldsymbol{\nu})_{i,2}, \sigma_{1}(\boldsymbol{\nu})_{i,3}$. The existence and uniqueness of parabolic structures follow by Remark \ref{remark}. The left diagram in (\ref{diag}) clearly commutes. We consider the right diagram. When a connection map $\nabla$ is of the form (\ref{normformeq}), we obtain
\begin{align*}
	\nabla'&=d+
	\begin{pmatrix}
		-\tfrac{1}{3}(\nu_0-\nu_1)& -a_{12}(1-z)& -a_{13}(1-z)\\
		-1& -(\nu_0+\nu_1)z+p+\nu_0-\tfrac{1}{3}(\nu_0-\nu_1)& 0\\
		0& z-(1-q)& -p-\tfrac{1}{3}(\nu_0-\nu_1)
	\end{pmatrix}
	\frac{dz}{z(z-1)}.
\end{align*}
We hence have 
\[
(\varphi\circ f_{\sigma_1})((E, \nabla, l_*))=(1-q: -p-\tfrac{1}{3}(\nu_0-\nu_1): 1)=\phi_{\sigma_1}((q: p: 1))=(\phi_{\sigma_1}\circ \varphi)((E, \nabla, l_*)).
\]
Since $\varphi\circ f_{\sigma_1}=\phi_{\sigma_1}\circ \varphi$ on a Zariski open subset of $M(\boldsymbol{\nu})$, two morphism $\varphi\circ f_{\sigma_1}$ and $\phi_{\sigma_1}\circ \varphi$ coincide on $M(\boldsymbol{\nu})$.
\end{proof}
Similarly, we can see that the action of $\sigma_2$ is realized by
\[
s_2^*(E, \nabla)\otimes (\mcO_{\pl}, d-\tfrac{\nu_0-\nu_2}{3}\tfrac{dz}{z}), 
\]
where $s_2\colon \pl \rightarrow \pl; z\mapsto \tfrac{1}{z}$.
\subsection{Action of $w_3$ on $\mcM$}

We finally consider the realization of $w_3$. We use middle convolution for logarithmic connections for realization. Here we mention the main theorem. The proof is in Section \ref{MTsec}. We will explain notations in the next section.

Let $\beta \in H^0(X, \Omega_X^1(\log J))$ with
\begin{equation}\label{betacond}
\beta^{H_i}=-\nu_{i,0}, \ \beta^{V_i}=\nu_{i,0}-\tfrac{2}{3}\beta^T, \  \beta^T=\nu_{0,0}+\nu_{1,0}+\nu_{\infty,0}=\gamma.
\end{equation}
Set
\[
\mcN^{00}:=\left\{(\nu_{i,j})\in \mcN^0\mid \nu_{i,j}-\nu_{i,k}\notin \mbZ\setminus\{0\}\ \text{for any $i$ and $j\neq k$}\right\}.
\]
The condition that $\nu_{i,j}-\nu_{i,k}\notin \mbZ\setminus\{0\}$ comes from Proposition \ref{variation}.
\begin{theorem}\label{MT}
For $\boldsymbol{\nu} \in \mcN^{00}$, the following is commutative;
\[
\begin{tikzcd}
	M(\boldsymbol{\nu})\arrow[r, "\mc_\beta"]\arrow[d, "\varphi_{\boldsymbol{\nu}}"]&M(w_3(\boldsymbol{\nu}))\arrow[d, "\varphi_{w_3(\boldsymbol{\nu})}"]\\
	S_{\boldsymbol{\nu}}\arrow[r, "\phi_{w_3}"]&S_{w_3(\boldsymbol{\nu})}.
\end{tikzcd}
\]
In particular, the morphism $f_{w_3}:=\mc_\beta$ commutes the diagrams in (\ref{diag}) when $\mcN^0$ is restricted to $\mcN^{00}$.
\end{theorem}

\section{Simpson's middle convolution}\label{mcsec}
We review the theory of middle convolution of logarithmic connections over $(\pl, D=[t_1]+\cdots+[t_n])$ defined in \cite{Sim}. 
Put $(Y, D_Y):=(\pl, D)$ and $(Z, D_Z):=(\pl, D)$, and let  $p_1\colon Y\times Z\rightarrow Y$ and $p_2\colon Y\times Z\rightarrow Z$ be the projections. Let $\pi \colon X \rightarrow Y\times Z$ be the blow-up at the $n$ points $(t_1, t_1), \ldots, (t_n, t_n)$, and set $\pi_1=p_1 \circ \pi$ and $\pi_2=p_2 \circ \pi$. We define a normal crossing divisor $J$ on $X$ as
\[
J:=\sum_{i=1}^nH_i+\sum_{i=1}^nV_i+\sum_{i=1}^nU_i+T,
\]
where $H_i$, $V_i$, $T$ are the strict transformations of $\{t_i\}\times Z$, $Y \times \{t_i\}$, the diagonal curve, respectively, and $U_i$ is the exceptional curve corresponding to $(t_i, t_i)$. Set $\xjzq{1}:=\xj{1}/\pi_2^*\zd$.

We consider the middle convolution $\mc_{\beta}(E, \nabla)$ of logarithmic connections $(E, \nabla)$ over $(\pl, \bm{t}=\{t_i\}_{i=1}^n)$ for $\beta \in H^0(X, \xj{1})$. We define a logarithmic connection $(F, \nabla_F)$ over $(X,J)$ as
\[
(F, \nabla_F)=\pi_1^*(E, \nabla)\otimes (\mcO_X, d+\beta). 
\]
Consider the de Rham complex $\DR_X:=\DR_X(F, \nabla_F)$ and the relative de Rham complex $\rDR:=\DR_{X/Z}(F, \nabla_F)$ defined as follows;
\begin{align*}
	\DR_X(F, \nabla_F) :=&[F \overset{\nabla_F}{\longrightarrow}F\otimes \xj{1} \overset{\nabla_F}{\longrightarrow}F\otimes \xj{2}],\\
	\DR_{X/Z}(F, \nabla_F) :=&[F \overset{\overline{\nabla}_F}{\longrightarrow}F\otimes \Omega^1_{(X,J)/(Z,D)}],
\end{align*}
where $\overline{\nabla}_F$ is the composition of $\nabla_F$ and the quotient map $F\otimes\xj{1}\rightarrow F\otimes\xjzq{1}$. Two complexes $\DR_X$ and $\rDR$ fit into the short exact sequence 
\begin{equation}\label{drex}
	0\to \rDR[-1]\otimes \pi_2^*\zd\to \DR_X\to \rDR\to 0.
\end{equation}
When $\mathbb{R}^1\pi_{2*}\rDR$ is locally free, the connecting homomorphism of the above short exact sequence defines a connection map
\[
\nabla^{\GM}: \mathbb{R}^1\pi_{2*}\rDR \to \mathbb{R}^2\pi_{2*}(\rDR[-1]\otimes \pi_2^*\zd) \cong \mathbb{R}^1\pi_{2*}\rDR \otimes \zd,
\]
which is called a logarithmic Gauss-Manin connection. We define the middle convolution $\mc_{\beta}(E,\nabla)$ as a subconnection of $(\mathbb{R}^1\pi_{2*}\rDR, \nabla^{\GM})$. Set
\[
\mcF^0:=(F|_{T}/\Image (\res_T\nabla_F)) \oplus \bigoplus_{i=1}^n(F|_{H_i}/\Image (\res_{H_i}\nabla_F)),
\]
where $\res_T\nabla_F\colon F|_{T}\rightarrow F|_{T}$ and $\res_{H_i}\nabla_{F}\colon F|_{H_i}\rightarrow F|_{H_i}$ be the endmorphisms induced by $\nabla_F$. We define a sheaf $F'$ by the following short exact sequence 
\begin{equation}\label{fprime}
	0 \to F' \to F\to \mcF^0\to 0. 
\end{equation}
Then, we have $\overline{\nabla}_F(F)\subset F'\otimes \xjzq{1}$. We define a complex $\rMDR:=\MDR_{X/T}(F, \nabla_F)$ of sheaves over $\mbC$ by
\begin{equation}
	\MDR_{X/T}(F, \nabla_F) :=[F \overset{\overline{\nabla}_F}{\longrightarrow}F'\otimes \Omega^1_{(X,J)/(Z,D)}].
\end{equation}
There is a short exact sequence
\[
0 \to \rMDR\to \rDR \to \mcF^0[-1] \to 0, 
\]
which induces 
\[\pi_{2*} \mcF^0[-1]\to \mbR^1\pi_{2*}\rMDR \to \mbR^1\pi_{2*}\rDR.
\]
Since $\pi_{2*} \mcF^0[-1]=0$, the sheaf $\mbR^1\pi_{2*}\rMDR$  becomes a subsheaf of $\mbR^1\pi_{2*}\rDR$. By definition, we have 
\[
\nabla^{\GM}(\mbR^1\pi_{2*}\rMDR)\subset \mbR^1\pi_{2*}\rMDR\otimes \xjzq{1}.
\]
Let  $\nabla^{\GM}_{\tmid}$ be the restriction of $\nabla^{\GM}$ on $\mbR^1\pi_{2*}\MDR_{X/T}$. 

\begin{definition}
Assume that $\mbR^1\pi_{2*}\rMDR$ is locally free. We call the pair $(\mbR^1\pi_{2*}\rMDR, \nabla^{\GM}_{\tmid})$ \textit{the middle convolution} of $(E, \nabla)$ for $\beta$, and denote it by $\mc_\beta(E,\nabla)$; 
\[
\mc_\beta(E,\nabla):=(\mbR^1\pi_{2*}\rMDR, \nabla^{\GM}_{\tmid}).
\]
\end{definition}

Next, we will see the variation of local exponents. Let us fix $\boldsymbol{\mu}=(\mu_{i,j})_{0\leq j\leq r-1}^{1\leq i\leq n}\in \mbC^{rn}$. For $\mu \in \mbC$, we define a divisor $[\boldsymbol{\mu}]_i$ of $\mathbb{C}$ by $[\boldsymbol{\mu}]_i:=\sum_{\mu}m_i(\mu)[\mu]$, where $m_i(\mu):=\#\{j\mid \mu_{i,j}=\mu\}$. We set
\[
\delta(\beta, \boldsymbol{\mu}):=(n-2)r-\sum_{i=1}^{n}m_i(-\beta^{H_i}).
\]
\begin{proposition}(Scholium 5.8. in \cite{Sim})\label{variation}
Given a $\boldsymbol{\mu}$-parabolic connection $(E, \nabla, l_*)$, set $(E', \nabla'):=\mc_{\beta}(E,\nabla)$. Let $\boldsymbol{\mu}'$ be the set of local exponents of $\nabla'$. Suppose that $\mu_{i,0}, \ldots, \mu_{i,r-1}$ don't differ by nonzero integers, $\mu+\beta^{H_i}+\beta^T\notin \mbZ$ and $\mu+\beta^{H_i}\notin \mbZ\setminus\{0\}$ for any local exponent of $\nabla'$ at $t_i$, and $\beta^T\notin \mbZ$. Then the rank of $E'$ is $r+\delta(\beta, \boldsymbol{\mu})$ and we have
\begin{align*}
	[\boldsymbol{\mu}']_i=
	&(m_i(-\beta^{H_i})+\delta(\beta, \boldsymbol{\mu}))[\beta^{V_i}]+\sum_{\mu+\beta^{H_i}\neq 0}m_i(\mu)[\mu+\beta^{U_i}].
\end{align*}
\end{proposition}
We denote the residue of $\beta\in H^0(X, \Omega_X^1(\log J))$ along $H_i, V_i, U_i, T$ by $\beta^{H_i}, \beta^{V_i}, \beta^{U_i}, \beta^{T}$. We can find by the residue formula for suitable curves that residues satisfy the relations
\begin{equation}\label{betarel}
\beta^T+\sum_{i=1}^n\beta^{H_i}=0, \quad  \beta^T+\sum_{i=1}^n\beta^{V_i}=0, \quad \beta^{U_i}=\beta^{H_i}+\beta^{V_i}+\beta^{T}.
\end{equation}
Conversely, given $\beta^{H_i}, \beta^{V_i}, \beta^{U_i}, \beta^{T} \in \mathbb{C}$ satisfying the relations $(\ref{betarel})$, there is a unique $\beta \in H^0(X, \Omega_X^1(\log J))$ with these residues. In particular, $\beta$ with the conditions (\ref{betacond}) uniquely exists.

\section{Proof of Main Theorem}\label{MTsec}

We will prove Theorem \ref{MT} in the following three steps.
\begin{itemize}\setlength{\itemsep}{0cm}
\item[Step 1.]  Find a local basis of $\mathbb{R}^1\pi_{2*}\MDR_{X/T}$. 
\item[Step 2.] Express $\nabla^{\GM}_{\tmid}$ for the above basis. 
\item[Step 3.] Write down the apparent singularity $\overline{q}$ and the dual parameter $\overline{p}$ by $q$ and $p$.
\end{itemize}

Let us fix $\boldsymbol{\nu} \in \mcN^{00}$. For $(E, \nabla, l_*)\in M(\boldsymbol{\nu})$, we set
\begin{align*}
(F, \nabla_F):=&\pi_1^*(E, \nabla_E)\otimes (\Opl, d+\beta),\\
(E', \nabla'):=&\mc_{\beta}(E ,\nabla)=(\mbR^1\pi_{2*}\rMDR, \nabla^{\GM}_{\tmid}).
\end{align*}
We will prove that $(E', \nabla')$ is made into a $w_3(\boldsymbol{\nu})$-parabolic connection in Proposition \ref{E'locexp}.
 
\subsection{Step1: Find a local basis of $G'=\mathbb{R}^1\pi_{2*}\rMDR$}

We take a basis $e_{0}, e_{1}, e_{2}$ (resp. $\tilde{e}_{0}, \tilde{e}_{1}, \tilde{e}_{2}$) of $E\cong \Opl(1)\oplus\Opl\oplus\Opl$ on $\pl\setminus \{0\}$ (resp. on $\pl\setminus \{\infty\}$) satisfying 
\begin{equation}\label{grelation}
e_{0}=\tilde{y}\tilde{e}_{0}, \quad  e_{1}=\tilde{e}_{1}, \quad e_{2}=\tilde{e}_{2}, 
\end{equation}
and
\[
\nabla(e_0, e_1, e_2)=(e_0, e_1, e_2)
\begin{pmatrix}
	0& a_{12}(y)& a_{13}(y)\\
	1& (\nu_0+\nu_1)z-p-\nu_0& 0\\
	0& y-q& p
\end{pmatrix}
\frac{dy}{y(y-1)}, 
\]
where $\tilde{y}=y^{-1}$, and $a_{12}(y), a_{13}(y)$ are the polynomials defined in the section \ref{normform}.

By the definition of $\rMDR$, there is an exact sequence
\begin{equation}\label{MDRex}
0 \to \pi_{2*}F\to\pi_{2*}(F'\otimes \Omega^1_{(X,J)/(Z,D)})\to E' \to\mathbb{R}^1\pi_{2*}F.
\end{equation}

\begin{lemma}\label{fpush}
$\pi_{2*}F\cong \mcO_Z^{\oplus 4}$ and $\mbR^1\pi_{2*}F=0$.
\end{lemma}
\begin{proof}
First we consider $\mbR^1\pi_{2*}F$. The fiber of $\mbR^1\pi_{2*}F$ at $t\in Z$ is naturally isomorphic to $H^1(X_t, F|_t)$.
When $t\in Z\setminus D_Z$, we have $(\mbR^1\pi_{2*}F)|_{t}\cong H^1(X_t, F|_t)\cong H^1(\pl, E)=0$. When $t\in D_Z$, there is a short exact sequence
\[
0\to \mcO_{U_t}(-w_t)\otimes F\to F|_t=\mcO_{X_t}\otimes F\to \mcO_{V_t}\otimes F\to 0
\]
where $w_t$ is the point in $U_t\cap V_t$.The above short exact sequence induces
\[
H^1(\pl, \Opl(-1)^{\oplus 3})\cong H^1(U_t, \mcO_{U_t}(-w_t)\otimes F)\longrightarrow H^1(X_t, F|_t)\longrightarrow H^1(V_t, \mcO_{V_t}\otimes F)\cong H^1(\pl, E).
\]
Since both sides are zero, we have $H^1(X_t, F|_t)=0$. We hence obtain $\mbR^1\pi_{2*}F=0$.

Next we consider $\pi_{2*}F$. We have $(\pi_{2*}F)|_{t}\cong H^1(X_t, F|_t)\cong H^0(\pl, E)$ for $t\in Z\setminus D_Z$. In particular, we find that  $\pi_{2*}F$ is generated by $e_{0}, ye_{0}, e_{1}, e_{2}$ (resp. $\tilde{e}_{0}, \tilde{y}\tilde{e}_{0}, \tilde{e}_{1}, \tilde{e}_{2}$) on $Z\setminus \{0\}$ (resp. $Z\setminus \{\infty\}$). Hence we have $\pi_{2*}F\cong \mcO_Z^{\oplus 4}$ by the relation (\ref{grelation}).
\end{proof}
We obtain $E' \cong\pi_{2*}(F'\otimes \Omega^1_{(X,J)/(Z,D)})/ \pi_{2*}F$ by the above lemma and the exact sequence (\ref{MDRex}). We can hence regard sections of $E'$ as equivalence classes of sections of $\pi_{2*}(F'\otimes \Omega^1_{(X,J)/(Z,D)})$. We will find a local basis of $\pi_{2*}(F'\otimes \Omega^1_{(X,J)/(Z,D)})$. The short exact sequence (\ref{fprime}) induces
\begin{equation}\label{fprimepush}
0 \to \pi_{2*}(F'\otimes \Omega^1_{(X,J)/(Z,D)}) \to \pi_{2*}(F\otimes \Omega^1_{(X,J)/(Z,D)})\to  \pi_{2*}\mcF^0.
\end{equation}
Since $\pi_1^*\nabla_{E}$ is holomorphic along $T$ and $\beta^T\neq 0$, we have $\Image(\res_T\nabla_{F})=\Image(\beta^T \id)=F|_T$. We hence get by definition the isomorphisms
\[
\mcF^0\cong \bigoplus_{i=0, 1, \infty}(E|_i/\Image(\res_{i}\nabla_{E}+\beta^{H_i}))\otimes \mcO_{H_i}
\]
and
\[
\pi_{2*}\mcF^0 \cong (\oplus_{i=0, 1, \infty}(E|_i/\Image(\res_{i}\nabla_{E}+\beta^{H_i})))\otimes \mcO_{Z}.
\]
By direct computation, we have
\begin{equation}\label{Image}
\begin{aligned}
&\Image (\res_0\nabla_{E}+\beta^{H_0})=
\mathbb{C}(\nu_{0,0}e_0+e_1)+\mathbb{C}((p+\nu_{0,1})(p+\nu_{0,2})e_0+qe_2),
\quad \\
&\Image (\res_1\nabla_{E}+\beta^{H_1})=
\mathbb{C}(-\nu_{1,0}e_0+e_1)+\mathbb{C}((p-\nu_{1,1})(p-\nu_{1,2})e_0+(q-1)e_2),\\
&\Image (\res_\infty\nabla_{E}+\beta^{H_\infty})=
\mathbb{C}((1+\nu_{\infty,0})e_0+e_1)+\mathbb{C}(-\nu_{\infty,1}\nu_{\infty,2} e_0+e_2).
\end{aligned}
\end{equation}
Hence, $\pi_{2*}\mcF^0$ is a trivial bundle of rank three.

A section of $\pi_{2*}(F\otimes \xjzq{1})$ on $Z\setminus \{\infty\}$ is written as
\[
(a_{00}e_0+a_{01}e_1+a_{02}e_2)\tfrac{dy}{y}+(a_{10}e_0+a_{11}e_1+a_{20}e_2)\tfrac{dy}{y-1}+(a_{0}e_0+a_{1}e_1+a_{2}e_2)\tfrac{dy}{y-z}+a_{\infty 0}e_0\otimes dy.
\]
We have the following lemma by (\ref{fprimepush}) and (\ref{Image}). 

\begin{lemma}\label{fprimegen}
$\pi_{2*}(F'\otimes \xjzq{1})$ is generated on $Z\setminus \{\infty\}$ by
\[\hspace{-6mm}
\begin{array}{ll}
\eta_{1}:=\nu_{0,0}e_{0}\otimes \tfrac{dy}{y}+e_{1}\otimes (\tfrac{dy}{y}-\tfrac{dy}{y-z}),  
&
\eta_{2}:=(p+\nu_{0,1})(p+\nu_{0,2})e_{0}\otimes \tfrac{dy}{y}+qe_{2}\otimes (\tfrac{dy}{y}-\tfrac{dy}{y-z}),  \\[4pt]
\eta_{3}:=-\nu_{1,0}e_{0}\otimes \tfrac{dy}{y-1}+e_{1}\otimes (\tfrac{dy}{y-1}-\tfrac{dy}{y-z}),  
&
\eta_{4}:=(p-\nu_{1,1})(p-\nu_{1,2})e_{0}\otimes \tfrac{dy}{y-1}+(q-1)e_{2}\otimes (\tfrac{dy}{y-1}-\tfrac{dy}{y-z}), \\[4pt]
\eta_{5}:=(1+\nu_{\infty,0})e_{0}\otimes dy+e_{1}\otimes \tfrac{dy}{y-z},  
&
\eta_{6}:=-\nu_{\infty,1}\nu_{\infty,2}e_{0}\otimes dy+
 e_{2}\otimes \tfrac{dy}{y-z}\\[4pt]
\eta_{7}:=e_{0}\otimes \tfrac{dy}{y-z}.
\end{array}
\]
\end{lemma}

\begin{proposition}\label{G'loc}
	$\pi_{2*}(F'\otimes \xjzq{1})\cong \mcO_Z(1)\oplus \mcO_Z^{\oplus 6}$ and $G'\cong \mcO_Z(1)\oplus \mcO_Z\oplus \mcO_Z$.
\end{proposition}
\begin{proof}
In the same manner of Lemma \ref{fprimegen}, we can see that $\pi_{2*}(F'\otimes \xjzq{1})$ is generated on $Z\setminus \{0\}$ by
\[\hspace{-6mm}
\begin{array}{ll}
\tilde{\eta}_{1}:=-\nu_{0,0}\tilde{e}_{0}\otimes d\tilde{y}-\tilde{e}_{1}\otimes\tfrac{d\tilde{y}}{\tilde{y}-\tilde{z}},  
&
\tilde{\eta}_{2}:=-(p+\nu_{0,1})(p+\nu_{0,2})\tilde{e}_{0}\otimes d\tilde{y}-q\tilde{e}_{2}\otimes\tfrac{d\tilde{y}}{\tilde{y}-\tilde{z}},    \\[4pt]
\tilde{\eta}_{3}:=-\nu_{1,0}\tilde{e}_{0}\otimes \tfrac{d\tilde{y}}{\tilde{y}-1}+\tilde{e}_{1}\otimes (\tfrac{d\tilde{y}}{\tilde{y}-1}-\tfrac{d\tilde{y}}{\tilde{y}-\tilde{z}}),  
&
\tilde{\eta}_{4}:=(p-\nu_{1,1})(p-\nu_{1,2})\tilde{e}_{0}\otimes \tfrac{d\tilde{y}}{\tilde{y}-1}+(q-1)\tilde{e}_{2}\otimes(\tfrac{d\tilde{y}}{\tilde{y}-1}-\tfrac{d\tilde{y}}{\tilde{y}-\tilde{z}}), \\[4pt]
\eta_{5}:=-(1+\nu_{\infty,0})\tilde{e}_{0}\otimes \tfrac{d\tilde{y}}{\tilde{y}}-\tilde{e}_{1}\otimes (\tfrac{d\tilde{y}}{\tilde{y}}-\tfrac{d\tilde{y}}{\tilde{y}-\tilde{z}}),  
&
\tilde{\eta}_{6}:=\nu_{\infty,1}\nu_{\infty,2}\tilde{e}_{0}\otimes \tfrac{d\tilde{y}}{\tilde{y}}-\tilde{e}_{2}\otimes (\tfrac{d\tilde{y}}{\tilde{y}}-\tfrac{d\tilde{y}}{\tilde{y}-\tilde{z}})\\[4pt]
\tilde{\eta}_{7}:=\tilde{e}_{0}\otimes \tfrac{d\tilde{y}}{\tilde{y}-\tilde{z}},
\end{array}
\]
where $\tilde{z}=z^{-1}$. These sections satisfy the following relations;
\[
\eta_{1}=\tilde{\eta}_{1}, \ \eta_{2}=\tilde{\eta}_{2}, \ \eta_{3}=\tilde{\eta}_{3}, \ \eta_{4}=\tilde{\eta}_{4}, \ \eta_{5}=\tilde{\eta}_{5}, \ \eta_{6}=\tilde{\eta}_{6},  \ \eta_{7}=\tilde{z}\tilde{\eta}_{7},
\]
which yields $\pi_{2*}(F'\otimes \xjzq{1}) \cong \mcO_Z(1)\oplus \mcO_Z^{\oplus 6}$. We note that $\mcO_Z(1)$ is generated by $\eta_7$.
	
 Now we have 
\[
G' \cong \Coker \left(\pi_{2*}\overline{\nabla}_F \colon \pi_{2*}F \rightarrow \pi_{2*}(F'\otimes \Omega^1_{(X,J)/(Z,D)})\right).
\]
 We see generators of $\Image \pi_{2*}\overline{\nabla}_F$. The sheaf $\pi_{2*}F$ is generated by $e_0, ye_0, e_1, e_2$ on $Z\setminus\{0\}$. By direct computation, we get
\begin{align*}
\overline{\nabla}_F(e_0)=&-\eta_1+\eta_3+\beta^T \eta_7,\\
\overline{\nabla}_F(ye_0)=&\eta_3+\eta_5+\beta^Tz \eta_7,\\
\overline{\nabla}_F(e_1)=&(p+\nu_0-\nu_{0,0})\eta_1+\eta_2-(p-\nu_{1,1}-\nu_{1,2})\eta_3-\eta_4-(1+\nu_{1,1}+\nu_{1,2}) \eta_5+\eta_6,\\
\overline{\nabla}_F(e_2)=&-\tfrac{p+\nu_{0,0}}{q}\eta_2+\tfrac{p-\nu_{1,0}}{q-1}\eta_4+\nu_{2,0}\eta_6. 
\end{align*}
Hence, we can take $[\eta_5], [\eta_6], [\eta_7]$ as a basis of $E'$ on $Z\setminus \{\infty\}$. Here, we denote the image of a section $\eta \in \pi_{2*}(F'\otimes \xjzq{1})$ by $[\eta]$.
 Therefore, we obtain $G'\cong \mcO_Z(1)\oplus \mcO_Z\oplus \mcO_Z$.
\end{proof}

We have the following by applying Proposition \ref{variation} to our setting. 
\begin{proposition}\label{E'locexp}
	Local exponents of $(E', \nabla')$ at $t_i$ are $\nu_{i,0}-\frac{2}{3}\gamma, \nu_{i,1}+\frac{1}{3}\gamma, \nu_{i,2}+\frac{1}{3}\gamma$. Hence, we can make $(E', \nabla')$ into a $w_3(\boldsymbol{\nu})$-parabolic connection by considering a parabolic structure $l'_*$, which is uniquely determined. In particular, the following diagram is commutative;
	\[
	\begin{tikzcd}
		\mcM\arrow[r, "\mc_\beta"]\arrow[d]&\mcM\arrow[d]\\
		\mcN^{00}\arrow[r, "w_3"]&\mcN^{00}.
	\end{tikzcd}
	\]
\end{proposition}

\subsection{Step2: Express $\nabla_{G'}$ for the above basis. }
The Gauss-Manin connection $\nabla^{\GM}\colon \mathbb{R}^1\pi_{2*}\DR_{X/Z} \rightarrow \mathbb{R}^1\pi_{2*}\DR_{X/Z} \otimes \Omega_Z^1(D_Z)$ is the connecting homomorphism induce by (\ref{drex}). We recall how to calculate the Gauss-Manin connection $\nabla^{\GM}$ by using \v{C}ech cohomology. We consider it over $Z_0=Z\setminus \{0\}$, that is, we consider the $\mathbb{C}$-linear map
\[
\nabla^{\GM}\colon \mbH^1(\pi^{-1}_2(Z_0), \DR_{X/Z}) \rightarrow \mbH^1(\pi^{-1}_2(Z_0), \DR_{X/Z}) \otimes \Omega_Z^1(D_Z)(Z_0).
\]
Through this subsection, the cohomology group $H^i(\mcF)$ means $H^i(\pi^{-1}_2(Z_0), \mcF)$. 

Let $\mcU:=\{U_i\}_i$ be an affine open covering of $\pi^{-1}_2(Z_0)$. For a bounded below complex of sheaves $\mcF^\bullet=(\{\mcF^n\}_n, \{d^n_{\mcF}\}_n)$, we define a complex $(\{\check{C}^n(\mcU, \mcF^\bullet)\}_n, \{d^n\}_n)$ by $\check{C}^n(\mcU, \mcF^\bullet)=\oplus_{p+q=n} \check{C}^{p}(\mcU, \mcF^q)$ and 
\[
d^n\colon \check{C}^n(\mcU, \mcF^\bullet)\longrightarrow \check{C}^{n+1}(\mcU, \mcF^\bullet); \  (\alpha_p)_p \longmapsto ((-1)^{p+1}d_{\mcF}^{q-1}(\alpha_{p+1})+d^p_q(\alpha_p))_p
\]
where $(\{C^i(\mcU, \mcF^q)\}_i, \{d^i_q\}_i)$ is the \v{C}ech complex of $\mcF^q$. It is known that the $i$-th hypercohomology group $\mbH^i(\mcF^{\bullet})$ is isomorphic to $\Ker(d^i)/\Image(d^{i-1})$.

By the definition of $\DR_{X}$ and $\DR_{X/Z}$, we have the commutative diagram
\begin{equation}\label{cechdiag}
\begin{tikzcd}
	H^0(F\otimes \xj{1})\arrow[r]\arrow[d]&H^0(F\otimes \xjzq{1})\arrow[d]\\
	\mbH^1(\DR_{X})\arrow[r]&\mbH^1(\DR_{X/Z}).
\end{tikzcd}
\end{equation}
Since $H^1(F)=0$ by Lemma \ref{fpush}, the right homomorphism is surjective.
Hence each element in $\mbH^1(\DR_{X/Z})$ is of the form $(0, \{\eta|_{U_i}\}_i)+\Image(\overline{\nabla}_F)$, where $\eta \in H^0(F\otimes \Omega^1_{(X,J)/(Z,D)})$. We note that $[\eta]=(0, \{\eta|_{U_i}\}_i)+\Image(d^0_{X/Z})$. We consider the image $\nabla^{\GM}([\eta])$. Since 
\[
H^1(F\otimes \pi_{2}^*\Omega^1_{Z}(D_Z))=H^1(F)\otimes \pi_{2}^*\Omega^1_{Z}(D_Z)(U_0)=0,
\]
the upper homomorphism in (\ref{cechdiag}) is surjective. Hence we can take a lift $\zeta \in H^0(F\otimes \xj{1})$ of $\eta$. Then, $(0, \{\zeta|_{U_i}\}_i)+\Image(d^0_{X}) \in \mbH^1(\DR_{X})$ is a lift of $[\eta]$. In particular,  the image $d^1_X(0, \{\zeta|_{U_i}\}_i)=(0, 0, \{\nabla_F(\zeta|_{U_i})\}_i)$ is in $H^0(F\otimes \xj{2})$.
Since there is a natural isomorphism 
\[
 H^0(F\otimes \xj{2})\cong H^0(F\otimes \xjzq{1})\otimes \Omega_Z^1(D_Z)(Z_0),
\]
 the section $\nabla_F(\zeta)$ is written in the form $\omega \wedge \tfrac{dz}{z(z-1)}$, where $\omega \in H^0(F\otimes \xjzq{1})$. Then, we have
\[
\nabla^{\GM}([\eta])=\nabla^{\GM}((0, \eta)+\Image(d^0_{X/Z}))=((0, \omega)+\Image(d^0_{X/Z}))\otimes \tfrac{dz}{z(z-1)}=[\omega]\otimes \tfrac{dz}{z(z-1)}.
\]

We apply this calculation to the basis $[\eta_5], [\eta_6], [\eta_7]$ of $E'$. Then, $\nabla'$ is written as
\begin{equation}\label{nablaprime}
\nabla'([\eta_7], [\eta_5], [\eta_6])
=([\eta_7], [\eta_5], [\eta_6])
\begin{pmatrix}
	*& *& *\\
	1& *& *\\
	0& c_{32}(z)& c_{33}(z)
\end{pmatrix}\frac{dz}{z(z-1)}.
\end{equation}

\subsection{Step3: Write down the apparent singularity $\overline{q}$ and the dual parameter $\overline{p}$}
Let $E'=E'_0\supset E'_1\supset E'_2\supset E'_3=0$ be the filtration in Proposition \ref{appfil}. Then, $E'_2$ is generated by $[\eta_{7}]$. Moreover, we can see that $E'_1$ is generated by $[\eta_{5}], [\eta_{7}]$ from (\ref{nablaprime}). Hence, the apparent singularity $\overline{q}$ of $(E', \nabla')$ is the zero of the linear polynomial $c_{32}(z)$, and the dual parameter $\overline{p}$ is the value $c_{33}(\overline{q})$. We can check that 
\[
\overline{q}=(f_{16}/f_{14})q, \quad \overline{p}=\frac{(\nu_{1,0}-(2/3)\beta^T)f_{14}f_{16}-(\nu_{2,0}-(2/3)\beta^T)f_{16}f_{46}-(\nu_{0,0}-(2/3)\beta^T)f_{14}f_{46}}{\beta^Tf_{14}},
\]
where $f_{ij}=f_{ij}(q, p, 1)$. Therefore, the transformation $\mc_\beta$ gives the rational map 
\[
\mathbb{P}^2 \dasharrow \mathbb{P}^2; (q:p:1)\mapsto (\overline{q}: \overline{p}: 1),
\]
which coincides with the map (\ref{quadtranf}). We hence have $\varphi_{w_3(\boldsymbol{\nu})}\circ \mc_\beta=w_3\circ \varphi_{\boldsymbol{\nu}}$ on a Zariski open subset of $M(\boldsymbol{\nu})$, and it holds on $M(\boldsymbol{\nu})$.


\appendix\section*{Appendix A: The action of $\widetilde{W}(E^{(1)}_6)$}

\subsection*{A.1 The action on $\Pic (S_{\boldsymbol{\nu}})$}
\[
\begin{array}{llll}
w_1(\mcE_i)=\mcE_{(23)(i)}, & w_2(\mcE_i)=\mcE_{(12)(i)}, &w_4(\mcE_i)=\mcE_{(78)(i)}, &w_5(\mcE_i)=\mcE_{(67)(i)}, \\
w_6(\mcE_i)=\mcE_{(59)(i)}, &w_0(\mcE_i)=\mcE_{(45)(i)}, &\sigma_{1}(\mcE_i)=\mcE_{(14)(25)(39)(i)}, &\sigma_{2}(\mcE_i)=\mcE_{(46)(57)(89)(i)}\\
\end{array}, 
\]
\[
w_3(\mcE_j)=
\left\{
\begin{array}{cc}
	\mcE_0-\mcE_4-\mcE_6 & j=1\\
	\mcE_0-\mcE_1-\mcE_6 & j=4\\
	\mcE_0-\mcE_1-\mcE_4 & j=6\\
	\mcE_j & j\neq 1, 4, 6.
\end{array}
\right.
\]

\subsection*{A.2 The action on $\mcN^0$}\label{actionN0}
\[
\begin{array}{ll}
	w_1(\nu_{i,j})=
	\left\{
	\begin{array}{cc}
		\nu_{i,(12)(j)} & i=1\\
		\nu_{i,j} & i\neq 1,
	\end{array}
	\right. &
	w_2(\nu_{i,j})=
	\left\{
	\begin{array}{cc}
		\nu_{i,(01)(j)} & i=1\\
		\nu_{i,j} & i\neq 1,
	\end{array}
	\right. \\
	w_4(\nu_{i,j})=
	\left\{
	\begin{array}{cc}
		\nu_{i,(12)(j)} & i=\infty\\
		\nu_{i,j} & i\neq \infty,
	\end{array}
	\right. &
	w_5(\nu_{i,j})=
	\left\{
	\begin{array}{cc}
		\nu^{(i)}_{(01)(j)} & i=\infty\\
		\nu_{i,j}& i\neq \infty,
	\end{array}
	\right. \\
	w_6(\nu_{i,j})=
	\left\{
	\begin{array}{cc}
		\nu_{i,(12)(j)} & i=0\\
		\nu_{i,j} & i\neq 0,
	\end{array}
	\right. &
	w_0(\nu_{i,j})=
	\left\{
	\begin{array}{cc}
		\nu_{i,(01)(j)} & i=0\\
		\nu_{i,j}& i\neq 0,
	\end{array}
	\right. \\
	w_3(\nu_{i,j})=
	\left\{
	\begin{array}{cc}
		\nu_{i,j}-\frac{2}{3}\gamma &  j=0\\
		\nu_{i,j}+\frac{1}{3}\gamma & j\neq 0\\
	\end{array}
	\right. ,\\
	\sigma_1(\nu_{i,j})=
	\left\{
	\begin{array}{cc}
		\nu_{1, j}+\frac{1}{3}(\nu_0-\nu_1) & i=0\\
		\nu_{0, j}+\frac{1}{3}(\nu_1-\nu_0) & i=1\\
		\nu_{2, j} & i=\infty,
	\end{array}
	\right.  
	&
	\sigma_2(\nu_{i,j})=
	\left\{
	\begin{array}{cc}
		\nu_{2, j}+\frac{1}{3}(\nu_0-\nu_2) & i=0\\
		\nu_{1, j} & i=1\\
		\nu_{0, j}+\frac{1}{3}(\nu_2-\nu_0) & i=\infty,
	\end{array}
	\right. 
\end{array}
\]

\end{document}